\documentclass{amsart}
\usepackage{sansmath}
\usepackage{amssymb,amsthm,amsmath}
\usepackage[numbers,sort&compress]{natbib}
\usepackage{color}
\usepackage{graphicx}
\usepackage{tikz}
\usepackage{caption}

\title[ ]{Anderson localization for multi-frequency quasi-periodic operators on $\mathbb{Z}^d$}

\author{Svetlana Jitomirskaya}
\address[S. Jitomirskaya]{Department of Mathematics, University of California, Irvine, CA, USA} \email{szhitomi@math.uci.edu}
\author{Wencai Liu}
\address[W. Liu]{Department of Mathematics, University of California, Irvine, CA, USA} \email{liuwencai1226@gmail.com}
\address{Current address: Department of Mathematics, Texas A\&M University, College Station, TX, USA}
\author{Yunfeng Shi}
\address[Y. Shi]{School of Mathematical Sciences, Peking University,
Beijing 100871,
China} \email{yunfengshi18@gmail.com}


\newcommand{\R}{\mathbb{R}}
\newcommand{\Z}{\mathbb{Z}}

\newcommand{\T}{\mathbb{T}}

\theoremstyle{plain}
\newtheorem{thm}{Theorem}[section]
 
 \newtheorem{lem}[thm]{Lemma}

 \newtheorem{defn}[thm]{Definition}
 \theoremstyle{remark}

 \numberwithin{equation}{section}

\begin{document}


\begin{abstract}
We establish Anderson localization for general analytic
$k$-frequency quasi-periodic operators on $\mathbb{Z}^d$ for
\textit{arbitrary} $k,d$.

\end{abstract}

\maketitle
\section{Introduction and main results}

The theory of quasiperiodic opetarors with analytic potentials has
seen dramatic advances in the last 20 years, since the development
of first non-perturbative methods for control of the Green's functions \cite{j94,j,bj,bbook,bg20}
that replaced earlier perturbation of eigenfunctions techniques. The most
well-developed and remarkably rich theory concerns the case of
one-dimensional one-frequency potentials, where powerful
reducibility/dynamical techniques are particularly enhanced by the
analyticity arguments. There are now non-perturbative
results on both small and high coupling sides (\cite{bbook,jmarx,you} and
references therein), global theory \cite{global}, and sharp arithmetic
transitions and related universality (e.g. \cite{ayz,jl1,jl2,jz}). However, if one increases either the
dimension of the undelying torus (the number of frequencies) or, especially, the
space dimension, the situation becomes significantly more
complicated. First, {\bf non-perturbative} { results} become, generally, false
\cite{bbook}, so throwing away small measure sets of parameters where
things actually do sometimes go bad, becomes a necessity. Even more
importantly,  one-dimensional (and therefore dynamical) techniques
are not applicable in higher space dimension. The fist
multi-dimensional localization was obtained by perturbative (KAM)
methods by Chulaevsky-Dinaburg for $k$-frequency operators on
$\ell^2(\Z^d)$ for $k=1$ and arbitrary $d$ \cite{cd}.
However, perturbative techniques have not been made to work to prove
localization in the multi-frequency case, $k>1,$ even for $d=1.$ Bourgain-Goldstein-Schlag developed a way to
apply some of the non-perturbative methods to the  two dimensional
case \cite{bgs}, obtaining localization at high coupling for
$k=d=2.$ This was extended
by Bourgain to arbitrary $k=d$ \cite{gafa}, where he developed a new
powerful scheme that allowed to circumvent the arithmetic difficulties
that restricted \cite{bgs} to $k=d=2.$ In this paper we extend
Bourgain's result to the case of general $k,d$ (in fact, an even
significantly more
general situation).

Let $S$ be a Toeplitz  (operator) matrix on $\ell^2(\Z^d)$ satisfying,
\begin{equation}\label{GO}
|S(n,n')|\leq   e^{-\rho|n-n'|},   \rho>0,
\end{equation}
 where $|n|:=\max\limits_{1\leq i\leq d}|n_i|$ for $n=(n_1,n_2,\cdots,n_d)\in \Z^d$.

Let $v$ be a real analytic function on $\mathbb{T}^b$, where $b=\sum\limits_{i=1}^{d}b_i$ ($b_i\in\mathbb{N}$ for $1\leq i\leq d$).

In this paper, we consider the following operators
\begin{equation}\label{op}
 H(x)=S+\lambda v(x+n\omega)\delta_{nn'},\ \ n,n'\in\mathbb{Z}^{d},
\end{equation}
where
\begin{eqnarray*}
&&x=(x_{11},\cdots,x_{b_11},\cdots,x_{1d},\cdots,x_{b_dd})\in\mathbb{T}^b,\\
&&n\omega=(n_1\omega_{11},\cdots,  n_1\omega_{b_11},\cdots,n_d\omega_{1d},\cdots,n_d\omega_{b_dd}).
\end{eqnarray*}

{\bf Example 0.}  Taking $b_i=1, i=1,...,d$ and the nearest neighbor Laplacian $S$ we
obtain operators considered in \cite{gafa}.

{\bf Example 1.}  $d=2$, $b_1=2,$ $ b_2=1$. $v$ is a function on $\T^3$.  For $x=(x_1,x_2,x_3)$ and $\omega=(\omega_1,\omega_2,\omega_3)$,
the operator becomes
\begin{equation}
 H=S+\lambda v(x_{1}+n_1\omega_1,x_2+n_1\omega_2,x_3+n_2\omega_3)\delta_{nn'},
\end{equation}
where $n=(n_1,n_2)$.

{\bf Example 2.} $b=kd, b_i=k, i=1,...,d,$ $f$ is a function on $\T^k$, and
$$v
(x_{11},\cdots,x_{k1},\cdots,x_{1d},\cdots,x_{kd})=f(x_{11}+\cdots+x_{1d},\cdots,x_{k1}+\cdots+x_{kd}).$$
Then the operator becomes
\begin{equation}\label{matrix}
 H(x) =S+\lambda f(x+nA)\delta_{nn'},
\end{equation}
where $x
\in \T^k, \,n\in\Z^d,$
and $A$ is a $d$ by $k$ matrix of frequencies. This is the most
general form of a $d$-dimensional quasiperiodic operator with a
$k$-dimensional phase space. The Aubry dual family has the form
\begin{equation}\label{dual}
 \tilde H(x) =F+\lambda s(x+Am)\delta_{nn'},
\end{equation}
where $x
\in \T^d, \,n\in\Z^k,$ and $F,S$ are Toeplitz operators with $(n,n')$ terms given by the
$n-n'$ Fourier coefficients of, correspondingly $f,s.$ The standard
Laplacian is therefore dual to the potential given by the sum of
cosines, and the dual of a general analytic potential is a Toeplitz
matrix as above.

{\bf Remark 1.} When considering families (\ref{matrix}) with $A$ restricted
to a linear submanifold of   $d$ by $k$ matrices of frequencies, one
needs to take $b_i$ equal to the number of free variables in the
$i^{th}$ row of $A$ and adjust $v$ accordingly. As such, the family
considered in \cite{gafa} can of course also be recast in this
language: it corresponds to $A$ restricted to $diag(\omega_1,...,\omega_d).$

 We call $x\in\mathbb{T}^b$ the phase, $\omega\in\mathbb{T}^b$ the
 frequency and $\lambda\geq 0$ the coupling.
Let $$x^j:=(x_{1j},\cdots,x_{b_jj})\in\mathbb{T}^{b_j}\  (1\leq j\leq d).$$
We assume $v$ satisfies the following \textit{non-degeneracy} condition:
 for  any $$(x^1,\cdots,x^{j-1},x^{j+1}\cdots,x^d)\in\mathbb{T}^{b-b_j},$$
the function $$\T^{b_j}\ni \theta \mapsto v(x^1,\cdots,x^{j-1},\theta,x^{j+1}\cdots,x^d)$$ is nonconstant.

Denote by ${\rm mes}$ the Lebesgue measure. We say operator $H$
satisfies Anderson localization if it has only pure point spectrum with
exponentially decaying eigenfunctions.

\begin{thm}\label{mthm}
Let $H(x)$ be given by (\ref{op}) with $v$ satisfying the
non-degeneracy condition. Then for any $\delta>0$, there is a $\lambda_0=\lambda_0(\delta,v,\rho,b,d)>0$ such that the following statement holds:
for any $\lambda\geq \lambda_0$ and any $x\in\mathbb{T}^{b}$, there exists  $\Omega=\Omega(x,\lambda v,\delta,\rho,b,d)\subset \mathbb{T}^{b}$ with $\mathrm{mes} (\mathbb{T}^{b}\setminus \Omega)\leq \delta $ such that for $\omega\in \Omega$, $H(x)$ satisfies Anderson localization.
\end{thm}
\newpage\clearpage
{\bf Remarks} \begin{enumerate}\item
In particular,  this holds for all operators (\ref{matrix}) with
arbitrary $k,d$ and any non-constant analytic function $f,$\footnote{As
  in Remark 1, the non-degenracy condition on $v$ leads to additional non-degeneracy conditions on $f$
  if the number of free variables in a certain row of the submanifold
  is bounded by $1.$ In particular, for $A$ restricted to
  $diag(\omega_1,...,\omega_d),$ as in \cite{gafa}, the required
  non-degeneracy condition is exactly as in \cite{gafa}.} which is an
important building block in the proof of absolutely continuous
spectrum for operators (\ref{dual}) \cite{bjp} and was the key initial
motivation for our work.
\item We note that our phase space dimension $b$ satisfies $b\geq d$ since $b=\sum_{i=1}^d b_i$, $b_i\geq 1$. This is essential
for our arguments. As shown in {\bf Example 2}, general quasiperiodic
operators always have $b\geq d$. However, operators \eqref{op} with $b<d$, for example $V_{n_1,n_2}=v(x+n_1\omega, x+n_2\omega)$, also appear naturally,
e.g. in the study of interacting particles, and our proof does not
apply in this setting. A localization result for a model with $b=1,
d=2$ was recently obtained by Bourgain-Kachkovskiy \cite{bk}.
\item Previous multidimensional/multifrequency localization results \cite{bgs,gafa} were
  not only restricted to $k=d$ ,  but also done only for the
  nearest neighbor Laplacian, i.e. $S(n,n')=\delta_{|n-n'|,1}.$
  The extension to general $S$ as in (\ref{GO}) is motivated by the
  Aubry duality purposes in \cite{bjp}. 
Localization for long-range operators (general $S$) was previously
obtained for $k=1$ in \cite{cd} and, nonperturbatively, for $k=d=1$ in \cite{bbook,bj}.
\end{enumerate}
The main scheme of our proof is definitely adapted from  Bourgain
\cite{gafa}. However, while our result is significantly more general and more
technically complex, our argument can also be viewed as both a
clarification and at the same time {\bf streamlining} of
\cite{gafa}. Indeed, our proof, while including more detail and hopefully
increasing the readability, is only
shorter  than the corresponding part of  Bourgain's. This is due to several important technical
improvements that we add to Bourgain's scheme. One important highlight
is that,  in the process of deterministic  multi-scale analysis proceeding
    from scale $N_1$ to $N_2$, a chain of scales between $N_1$ and
    $N_2$ has  always   been used in the past work,  \cite{gafa,bgs}. Here,
  instead of gluing   ``good"  Green's function at multiple  scales
  between $N_1$ and  $N_2$  to establish the ``goodness" of Green's
  function at scale $N_2$, we  find a way to directly use the ``good"
  Green's function at scale   $N_1$ + subexponential  bound of  the
  norm  to prove the ``goodness" of Green's function at scale $N_2$.

Another issue we want to highlight is that the $k=d=2$ analysis of
\cite{bgs} required dealing with many different types of elementary
regions, something that would be prohibitively difficult to carry out
in higher dimensions. In dealing with higher dimensions in
\cite{gafa} Bourgain
significantly reduces the allowed elementary regions. This
comes at the price of some complications in dealing  with the  lattice points at
 the boundary of the elementary regions, which Bourgain
 claims can be carried out, but provides no detail. We use the same
 (slightly corrected)
 type of
 restriction on the elementary regions but believe this
 issue is not entirely trivial and tackling it requires a certain modification of
 the procedure, which we provide in full detail.


Non-perturbative proofs of localization for $d>1$ are in a sense a
version of deterministic multi-scale analysis. The latter is a
powerful method originally developed for random operators by
Fr\"ohlich and Spencer \cite{fs}, that crucially relied on independence and
Wegner's Lemma that is effectively dependent on rank-one
perturbations. For the deterministic version, difficulties with lack of
independence/rank one perturbations are circumvented by the semi-algebraic sets
considerations and subharmonicity arguments \cite{bbook}. The
non-perturbative proofs consist of two parts. First,
one needs to obtain measure and complexity estimates for
phases/frequencies with exponential off-diagonal decay and
subexponential upper bounds for the matrix
elements of the Green's function for box-restricted operators for
a given energy. From this, localization follows through elimination of
energy via an argument involving complexity bounds on semi-algebraic
sets. The second part is by now rather standard and follows the
reasonably short argument in \cite{gafa} essentially verbatim. In fact,
it is the first part that presents the main difficulty associated with
higher dimensions. Thus we
focus only on the first, single energy, part here. This is also where
the key difficulty in extending \cite{bgs} and the key difference
between \cite{gafa} and \cite{bgs} lies. One needs to
guarantee a sublinear upper bound on the number of
times the ergodic trajectory hits certain forbidden regions of given
measure/algebraic complexity, without further detail on the structure
of those fordbidden regions. A key argument in \cite{bgs} is a Lemma
that does guarantee it for  $k=b=2$   under an explicit arithmetic
condition on the frequencies. Roughly, it means that too many points
on the trajectory of rotation close to an algebraic curve of a bounded
degree would force it to oscillate more than the degree
allows. However, this statement is not extendable to $d\geq 3.$
In \cite{gafa} Bourgain instead developed a way to restrict to suitable
frequencies already for the first step, which turned out to be a very
robust approach that we also develop here. Besides the elimination of energy argument, we do not include detailed
proofs of two further statements very similar to those in \cite{gafa},
and with proofs presented there in a very clear way. The proofs  that are similar to Bourgain's  that we do
present either have certain novelty or contain important technical
clarifications.

In Section 2 we introduce the main concepts and also list the above
mentioned results for which we do not present detailed proofs. One
such concept is ``property P at scale $N$'' - essentially, the single
energy statement one wants to establish for all large scales, that allows to streamline
certain formulations.  Section 3 is devoted to the main multi-scale
argument: property P at scales $N, N^c$ implies property P at an
interval of subexponentially large scales, Theorem \ref{mul1}. In section 4 we take
care of the initial scale and give a very short argument to obtain the
final single energy estimate, Theorem \ref{ldt}, from Theorem
\ref{mul1}. In the appendix we prove a several variables matrix-valued
Cartan estimate (Lemma \ref{mcl}  used in the proof of Theorem \ref{mul1}), that follows Bourgain's one-variable argument in
\cite{bbook} but also uses high-dimensional Cartan sets estimates
of \cite{gs}.

\section{Preparations}
\subsection{Notation}

For any $x\in\mathbb{R}^{d_1}$ and $X\subset\mathbb{R}^{d_1+d_2}$,  denote  the $x$-section of $X$ by
$$X(x):=\{y\in\mathbb{R}^{d_2}:\  (x,y)\in X\}.$$
Let $\tilde{b}=\max_i b_i$.
For any $x\in\mathbb{T}^b$ and $1\leq j\leq d$, let ${x}_j^\neg=(x^1,\cdots,x^{j-1},x^{j+1}\cdots,x^d)\in\mathbb{T}^{b-b_j}$.
For $x=(x_1,x_2,\cdots,x_l),y=(y_1,y_2,\cdots,y_l)\in \R^{l}$, let $|x-y|=\max_{i}|x_i-y_i|$.

For $\Lambda_1, \Lambda\subset\mathbb{Z}^d$,  we introduce
$$\mathrm{diam}(\Lambda)=\sup_{n,n'\in \Lambda}|n-n'|, \ \mathrm{dist}(m,\Lambda)=\inf_{n\in \Lambda}|m-n| \ (m\in\mathbb{R}^d),$$
and $\mathrm{dist}(\Lambda_1,\Lambda)=\inf\limits_{n\in \Lambda_1}\mathrm{dist}(n,\Lambda)$.

We also use $\|\cdot\|$ as $\ell^2$ norm of the matrix. For
convenience,  in the following, we study operator $\lambda^{-1}H(x) $.
We always assume $\lambda>1$. Since the spectra of $\lambda^{-1}H(x)$
are bounded by $C(S,v)$,
we can further assume $E$ is bounded.

\subsection{Green's functions and elementary regions}
For $\Lambda\subset\mathbb{Z}^d$,  let $R_{\Lambda}$ be the restriction operator, i.e.,  $(R_{\Lambda}\xi)(n)=\xi(n)$ for $n\in \Lambda$, and
$(R_{\Lambda}\xi)(n)=0$ for $n\notin \Lambda$. Denote by  $H_{\Lambda}=R_{\Lambda} HR_{\Lambda}$ and the Green's functions
\begin{equation*}
  G_{\Lambda}(E;x)=(R_{\Lambda} (\lambda^{-1}H-E+i0)R_{\Lambda})^{-1}.
\end{equation*}
We will also write $G_{\Lambda}$ when there is no ambiguity.
Clearly,
\begin{equation}\label{GGshift}
  G_{n+\Lambda}(x)=G_{\Lambda}(x+n\omega).
\end{equation}
We denote by $Q_N$ an elementary region of size $N$ centered at 0, which is one of the following regions,
\begin{equation*}
  Q_N=[-N,N]^d
\end{equation*}
or
$$Q_N=[-N,N]^d\setminus\{n\in\mathbb{Z}^d: \ n_i\varsigma_i 0, 1\leq i\leq d\},$$
where  for $ i=1,2,\cdots,d$, $ \varsigma_i\in \{<,>,\emptyset\}^{d}$ and at least two $ \varsigma_i$  are not $\emptyset$.

Denote by $\mathcal{E}_N^{0}$ the set of all elementary regions of size $N$ centered at 0. Let $\mathcal{E}_N$ be the set of all translates of  elementary regions, namely,
$$\mathcal{E}_N:=\{n+Q_N\}_{n\in\mathbb{Z}^d,Q_N\in \mathcal{E}_N^{0}}.$$


\subsection{Semi-algebraic sets}
\begin{defn}[Chapter 9, \cite{bbook}]
A set $\mathcal{S}\subset \mathbb{R}^n$ is called a semi-algebraic set if it is a finite union of sets defined by a finite number of polynomial equalities and inequalities. More precisely, let $\{P_1,\cdots,P_s\}\subset\mathbb{R}[x_1,\cdots,x_n]$ be a family of real polynomials whose degrees are bounded by $d$. A (closed) semi-algebraic set $\mathcal{S}$ is given by an expression
\begin{equation}\label{smd}
\mathcal{S}=\bigcup\limits_{j}\bigcap\limits_{\ell\in\mathcal{L}_j}\left\{x\in\mathbb{R}^n: \ P_{\ell}(x)\varsigma_{j\ell}0\right\},
\end{equation}
where $\mathcal{L}_j\subset\{1,\cdots,s\}$ and $\varsigma_{j\ell}\in\{\geq,\leq,=\}$. Then we say that $\mathcal{S}$ has degree at most $sd$. In fact, the degree of $\mathcal{S}$ which is denoted by $\deg(\mathcal{S})$, means the  smallest $sd$ over all representations as in (\ref{smd}).
\end{defn}

In \cite{gafa}, Bourgain proved a  result for eliminating several variables.
\begin{lem}[Lemma 1.18, \cite{gafa}]\label{svl}
Let $\mathcal{S}\subset [0,1]^{d+r}$ be a semi-algebraic set of degree $B$ and such that
$$\mathrm{mes}(\mathcal{S}(y))<\eta\ \mathrm{for}\ \forall\  y\in [0,1]^r.$$
Then the set
\begin{equation*}
\left\{(x_1,\cdots,x_{2^r})\in [0,1]^{d2^r}:\  \bigcap\limits_{1\leq i\leq 2^r}\mathcal{S}(x_i)\neq\emptyset\right\}
\end{equation*}
is semi-algebraic of degree at most $B^{C}$ and measure at most
\begin{equation*}
B^{C}\eta^{d^{-r}2^{-r(r-1)/2}},
\end{equation*}
where $C=C(d,r)>0$.
\end{lem}
Another important fact is the following decomposition Lemma for semi-algebraic sets in the product spaces.
\begin{lem}[\cite{gafa,bbook}]\label{projl}
Let $\mathcal{S}\subset[0,1]^{d=d_1+d_2}$ be a semi-algebraic set of degree $\deg(\mathcal{S})=B$ and $\mathrm{mes}_d(\mathcal{S})\leq\eta$, where
\begin{equation}
\log B\ll \log\frac{1}{\eta},
\end{equation}
with
$$ \eta^{\frac{1}{d}}\leq\epsilon.$$
Then there is a decomposition of $\mathcal{S}$ as
$$\mathcal{S}=\mathcal{S}_1\cup\mathcal{S}_2$$
such that the projection of $\mathcal{S}_1$ on $[0,1]^{d_1}$ has small measure
$$\mathrm{mes}_{d_1}(\mathrm{Proj}_{[0,1]^{d_1}}\mathcal{S}_1)\leq B^{C(d)}\epsilon,$$
and $\mathcal{S}_2$ has the transversality property
$$\mathrm{mes}_{d_2}(\mathcal{L}\cap \mathcal{S}_2)\leq B^{C(d)}\epsilon^{-1}\eta^{\frac{1}{d}},$$
where $\mathcal{L}$ is any $d_2$-dimensional hyperplane in $[0,1]^d$ s.t.,
$$\max\limits_{1\leq j\leq d_1}|\mathrm{Proj}_\mathcal{L}(e_j)|<{\epsilon},$$
where we denote by $e_1,\cdots,e_{d_1}$ the  coordinate vectors in $\R^{d_1}$.
\end{lem}
We then have
\begin{lem}\label{rmf}
Suppose that $\omega^i\in \R^{l_i}$ ($i=1,2,3,\cdots,r$) and $l=\sum_{i=1}^r l_i$.
Let  $\mathcal{S}\subset [0,1]^{l J}$ be a semi-algebraic set of degree $B$ and such that
$$\mathrm{mes}(\mathcal{S})<\eta.$$
For  $\omega= (\omega^1,\cdots,\omega^r)\in[0,1]^{l}$ and $ n=(n_1,n_2,\cdots,n_r)\in\mathbb{Z}^r$, define
$$n\omega=(n_1\omega^1,n_2\omega^2,\cdots, n_r\omega^r).$$
Let $\mathcal{N}^1,\cdots,\mathcal{N}^{J-1}\subset \mathbb{Z}^r$ be finite sets with the following property
$$\min\limits_{1\leq s\leq r}|n_s|>(B\max\limits_{1\leq s\leq r}|m_s|)^C, $$
where $n\in \mathcal{N}^{i}, m\in\mathcal{N}^{i-1}\  (2\leq i\leq J-1)$, where $C=C(J,l)$.
 Assume also
\begin{equation}\label{neta}\max\limits_{n\in\mathcal{N}^{J-1}}|n|^C<\frac{1}{\eta}.\end{equation}
Then
\begin{equation*}
\mathrm{mes}(\{\omega\in [0,1]^{l}:\  \exists \ n^{(i)}\ \in\mathcal{N}^i\  s.t.,\  (\omega,n^{(1)}\omega,\cdots,n^{(J-1)}\omega)\mod\mathbb{Z}^{lJ}\in \mathcal{S}\})\leq B^{C}\delta,
\end{equation*}
where
$$\delta^{-1}=\min\limits_{n\in\mathcal{N}^1}\min\limits_{1\leq s\leq r}|n_s|.$$
\end{lem}
\begin{proof}
 The proof follows  from Lemmas \ref{svl} and \ref{projl} just as  the
 proof of  Lemma 1.20 in \cite{gafa}  follows from the corresponding
 Lemma  1.18  and property (1.5) of semi-algebraic sets in \cite{gafa}.
\end{proof}
\begin{defn}
We say $(E,x)$ is $(\bar{\rho},N)$ good, if for any $Q_N\in\mathcal{E}_N^{0}$,
\begin{eqnarray}
\label{ldt1}&& \|G_{Q_N}(E;x)\|\leq e^{\sqrt{N}},\\
\label{ldt2}&& |G_{Q_N}(E;x)(n,n')|\leq  e^{-\bar{\rho}|n-n'|}\  {\mathrm{for} \ |n-n'|\geq \frac{N}{10}}.
\end{eqnarray}
\end{defn}
\begin{defn}
We say Green's function  satisfies property $P$ with parameters  $(\gamma,\bar{\rho})$ at size $N$ if
 there is a semi-algebraic set $\Omega_N=\Omega_N(\lambda v,\rho,b,d) \subset \mathbb{T}^b$ with  $\deg(\Omega_N)\leq N^{4d}$ such that the following statement is true:
  for any   $\omega\in\Omega_N$ and $ E\in\mathbb{R}$,   there exists a  set $X_N=X_N(\lambda v,\rho,b,d,\omega,E)\subset \mathbb{T}^b$  such that
 \begin{equation}\label{Gstar}
 \sup_{1\leq j\leq d,x_j^\neg\in \T^{b-b_j}}\mathrm{mes}( X_N(x_j^\neg))\leq e^{-N^{\gamma}},\end{equation}
and for any $x$ not in $  X_N$, $(E,x)$ is $(\bar{\rho},N)$ good.
\end{defn}

\begin{thm}\label{mul}
There exist small positive constants $c_3<c_4<1$, where $c_3$ and $c_4$ depend on $b,d$ such that the following statements are true.
Let $c_1=\frac{c_3}{4\tilde{b}}$ and $c_2=c_1^2/2$.  Fix a large number $N_1$. Let $N_2= N_1^{2/c_1}$ and $N_3=  e^{N_2^{c_2}}$.
Suppose the Green's functions satisfy property P at size $N_1$  with parameters  $(c_1,\bar{\rho})$, and corresponding semi-algebraic sets $\Omega_{N_1}$.
 Then there exists a semi-algebraic set $\Omega_{3} \subset \Omega_{{N}_1}$ with  $\deg(\Omega_{ 3})\leq {N}_3^{4d}$ and
$\mathrm{mes}((\Omega_{{N}_1}\backslash\Omega_{3})\leq {N_3}^{-c_3}$
such that, if $\omega\in\Omega_{3}$, then for any $E\in \R$ and $x\in \T^b$, there exists  ${N}_3^{c_3}< N< {N}_3^{c_4}$ such that,
for all  $k\in \Lambda \backslash \bar\Lambda$,  $ x+k\omega \mod \Z^b\notin X_{N_1}$, where
\begin{equation*}
  \Lambda=[-N,N]^d,\bar\Lambda=[-N^{\frac{1}{10d}},N^{\frac{1}{10d}}]^{d}.
\end{equation*}

\end{thm}


\begin{proof}
The proof is based on Lemmas \ref{svl} and \ref{rmf}. For details, we refer the reader to the proof of the \textsc{Claim} in  \cite[p.694]{gafa}.
To make it easier to check the corresponding relation between  Theorem \ref{mul} and Claim in \cite{gafa}, we present the alignment of our notations with these of \cite{gafa}.
Let  $X(B)$ denote the notation $X$ used  in  \cite{gafa}.
\begin{enumerate}
\item $\tilde{b}(B)=1$ since $b_i(B)=1$, $i=1,2,\cdots, d$.
\item $c_1=c_1(B)$, $c_2=c_2(B)$, $c_3=c_5(B)$ and $c_4=c_4(B)$. The formula before (2.8)  in \cite{gafa} gives  the relation between
$ c_1$ and $c_2$. The relation between   $ c_1$ and $c_3$ is presented at the end of  Section 2 in \cite{gafa}.
  \item  $N_1=N_2(B)$, $N_3=\bar{N}(B)$  and ${N}_3^{c_3}=\bar{\bar{N}}(B)$. See (2.8), (2.11) and (2.24) in \cite{gafa} for the corresponding relations.
  \item  $\Omega_{3}=\Omega_{\bar{N}}(B)$. See (2.25) in \cite{gafa} .
\end{enumerate}
\end{proof}


\section{Resolvent identities and Cartan's Lemma}

Let $\Lambda_1,\Lambda_2\subset \Z^d$ and $\Lambda_1\cap\Lambda_2=\emptyset$. Let $\Lambda=\Lambda_1\cup \Lambda_2$.
 Suppose that $ R_{\Lambda}(\lambda^{-1}{H}(x)-E)R_{\Lambda}$ and $ R_{\Lambda_i}(\lambda^{-1}{H}(x)-E)R_{\Lambda_i}$, $i=1,2$ are invertible.
 Then
 \begin{equation*}
   G_{\Lambda}=G_{\Lambda_1}+G_{\Lambda_2}-{\lambda}^{-1}(G_{\Lambda_1}+G_{\Lambda_2})( {H}_{\Lambda}-{H}_{\Lambda_1}-{H}_{\Lambda_2})G_{\Lambda}.
 \end{equation*}
 If $m\in \Lambda_1$ and $n\in \Lambda$, we have
 \begin{equation}\label{Greso}
    |G_{\Lambda}(m,n)|\leq |G_{\Lambda_1}(m,n)|\chi_{\Lambda_1}(n)+{\lambda}^{-1}\sum_{n^{\prime}\in \Lambda_1,n^{\prime\prime}\in \Lambda_2} e^{-\rho|n^{\prime}-n^{\prime\prime}|}|G_{\Lambda_1}(m,n^{\prime})||G_{\Lambda}(n^{\prime\prime},n)|.
 \end{equation}
 We remind
\begin{lem}[Schur  test]\label{schur}
Suppose $A=A_{ij}$ is a   symmetric matrix. Then
\begin{equation*}
  \|A\|\leq \sup_{i}\sum_{j}|A_{ij}|.
\end{equation*}
\end{lem}
We now prove

\begin{lem}\label{res1}
 Let $ M_0\geq (\log N)^{2},\bar\rho\in [\frac{\rho}{2},\rho]$ and $M_1\leq N$. Let ${\rm diam}(\Lambda)\leq 2N+1$. Suppose that for any $n\in \Lambda $, there exists some  $ W=W(n)\in \mathcal{E}_M$ with
$M_0\leq M\leq M_1$ such that
$n\in W\subset \Lambda$,  ${\rm dist} (n,\Lambda \backslash W)\geq \frac{M}{2}$ and
\begin{eqnarray}
\label{w1}&& \|G_{W(n)}(E;x)\|\leq2 e^{\sqrt{M}},\\
\label{w2}&& |G_{W(n)}(E;x)(n,n')|\leq  2e^{-\bar\rho|n-n'|}\  {\mathrm{for} \ |n-n'|\geq \frac{M}{10}}.
\end{eqnarray}
  We assume further that $N$ is large enough so that
\begin{equation}\label{ml}
\sup_{M_0\leq M\leq M_1}2 \lambda^{-1}e^{\sqrt{M}}(2M+1)^{d}e^{-\frac{3\rho}{20}M}\sum_{j=0}^{\infty}(2j+1)^de^{-\frac{\rho}{2}j}\leq \frac{1}{2}.
\end{equation}
Then
\begin{equation*}
  \|G_{\Lambda}(E;x)\|\leq 4 (2M_1+1)^d e^{\sqrt{M_1}}.
\end{equation*}
\end{lem}
\begin{proof}
For simplicity, we drop the dependence on $E$ and $x$.
Under the assumption of \eqref{ml}, it is easy to check that for all $M_0\leq M\leq M_1$,
\begin{equation}\label{Large1}
 2 {\lambda}^{-1}(2M+1)^de^{\sqrt{M}+\frac{\rho}{10}M}\sum_{n_2\in \Lambda\atop |n_2-n|\geq \frac{M}{2}}   e^{-\frac{\rho}{2}|n-n_2|}\leq \frac{1}{2}.
\end{equation}

By \eqref{w1} and \eqref{w2},
one has
\begin{equation}\label{Ugood}
   |G_{W(n)}(n,n')|\leq 2e^{\sqrt{M}+\frac{\bar{\rho}}{10}M}e^{-\bar\rho|n-n'|}.
\end{equation}
For each $n\in \Lambda$, applying \eqref{Greso} with $\Lambda_1=W(n)$, one has
\begin{equation}\label{BGSle}
  |G_{\Lambda}(n,n')|\leq |G_{W(n)}(n,n')|\chi_{W(n)}(n')+{\lambda}^{-1}\sum_{n_1\in W(n)\atop n_2\in \Lambda\backslash W(n)} e^{-\rho|n_1-n_2|}|G_{W(n)}(n,n_1)||G_{\Lambda}(n_2,n')|.
\end{equation}
By \eqref{Ugood} and the fact that $|W(n)|\leq (2M+1)^d$, one has
\begin{eqnarray}
  |G_{\Lambda}(n,n')| &\leq&  |G_{W(n)}(n,n')|\chi_{W(n)}(n')+2{\lambda}^{-1}\sum_{n_1\in W(n)\atop n_2\in \Lambda\backslash W(n)} e^{\sqrt{M}+\frac{\bar{\rho}}{10}M}e^{-\bar\rho|n-n_1|} e^{-\rho|n_1-n_2|}|G_{\Lambda}(n_2,n')|\nonumber  \\
  &\leq&|G_{W(n)}(n,n')|\chi_{W(n)}(n')+2{\lambda}^{-1}(2M+1)^de^{\sqrt{M}+\frac{\rho}{10}M}\sum_{n_2\in \Lambda\backslash W(n)}   e^{-\frac{\rho}{2}|n-n_2|}|G_{\Lambda}(n_2,n')| \nonumber\\
  &\leq&|G_{W(n)}(n,n')|\chi_{W(n)}(n')+2{\lambda}^{-1}(2M+1)^de^{\sqrt{M}+\frac{\rho}{10}M}\sum_{n_2\in \Lambda\atop |n_2-n|\geq \frac{M}{2}}   e^{-\frac{\rho}{2}|n-n_2|}|G_{\Lambda}(n_2,n')|
  \label{BGSle1}.
\end{eqnarray}
where the last inequality holds by the assumption ${\rm dist} (n,\Lambda\backslash W(n))\geq \frac{M}{2}$.

Summing over $n'\in \Lambda$ in \eqref{BGSle1} and noticing \eqref{Large1} yields
\begin{eqnarray}
 \sup_{n\in \Lambda}\sum_{n'\in \Lambda} |G_{\Lambda}(n,n')|
  &\leq&2(2M_1+1)^de^{\sqrt{M_1}}+\frac{1}{2}\sup_{n_2\in \Lambda}\sum_{n'\in \Lambda}|G_{\Lambda}(n_2,n')|.
\end{eqnarray}

Now the lemma follows from  Lemma \ref{schur}.

\end{proof}

\begin{thm}\label{res2}
Let ${\rm diam}(\Lambda)\leq 2N+1$ and $ {\rm diam}(\Lambda_1)\leq N^{\frac{1}{2d}}$. Let $ M_0\geq (\log N)^{2},\bar\rho\in [\frac{\rho}{2},\frac{4\rho}{5}]$. Suppose that for any $n\in \Lambda \backslash\Lambda_1$, there exists some  $ W=W(n)\in \mathcal{E}_M$ with
$M\geq M_0$ such that
$n\in W$,  ${\rm dist} (n,\Lambda\backslash \Lambda_1\backslash W)\geq \frac{M}{2}$, $W\subset \Lambda\backslash\Lambda_1$  and
\begin{eqnarray*}
&& \|G_{W}(E;x)\|\leq e^{\sqrt{M}},\\
&& |G_{W}(E;x)(n,n')|\leq  e^{- \bar{\rho}|n-n'|}\  {\mathrm{for} \ |n-n'|\geq \frac{M}{10}}.
\end{eqnarray*}
Suppose that
\begin{equation*}
\|G_{\Lambda}(E;x)\|\leq e^{\sqrt{N} }.
\end{equation*}
Then
\begin{eqnarray*}
 |G_{\Lambda}(E;x)(n,n')|\leq e^{-(\bar{\rho}-\frac{O(1)}{M_0^{1/2}})|n-n'|}\ \mathrm{for}\  |n-n'|\geq\frac{N}{10}.
\end{eqnarray*}
\end{thm}
\begin{proof}
As usual, we drop the dependence on $E$ and $x$.
Suppose $|n-n^{\prime}|\geq N^{\frac{1}{d}}+1$. Obviously, one of $n$ and $n^{\prime}$ is not in $\Lambda_1$.
By the self-adjointness of Green's functions, we can assume $n\notin \Lambda_1$.

Applying \eqref{Greso} with $\Lambda_1=W=W(n)$, one has
\begin{equation}\label{BGSle}
  |G_{\Lambda}(n,n')|\leq {\lambda}^{-1}\sum_{n_1\in W\atop n_2\in \Lambda\backslash W} e^{-\rho|n_1-n_2|}|G_{W}(n,n_1)||G_{\Lambda}(n_2,n')|.
\end{equation}
It implies (since $\lambda>1$)
\begin{eqnarray}
\nonumber|G_{\Lambda}(n,n')|
&\leq&  \sum_{n_1\in W,|n_1-n|\leq \frac{M}{10}-1\atop n_2\in \Lambda\backslash W} e^{-\rho|n_1-n_2|}|G_{W}(n,n_1)||G_{\Lambda}(n_2,n')|\\
\nonumber&&+ \sum_{n_1\in W,|n_1-n|\geq \frac{M}{10}\atop n_2\in \Lambda\backslash W} e^{-\rho|n_1-n_2|}|G_{W}(n,n_1)||G_{\Lambda}(n_2,n')|\\
&\leq& \nonumber  \sum_{n_1\in W,|n_1-n|\leq \frac{M}{10}-1\atop n_2\in \Lambda\backslash W} e^{\sqrt{M}}e^{-\rho|n_1-n_2|}|G_{\Lambda}(n_2,n')|\\
\nonumber&&+ \sum_{n_1\in W,|n_1-n|\geq \frac{M}{10}\atop n_2\in \Lambda\backslash W} e^{-\rho|n_1-n_2|} e^{-\bar{\rho} |n-n_1|}|G_{\Lambda}(n_2,n')|\\
&\leq& \nonumber \sum_{n_1\in W,|n_1-n|\leq \frac{M}{10}-1\atop n_2\in \Lambda\backslash W} e^{\sqrt{M}}e^{-\bar\rho|n-n_2|}|G_{\Lambda}(n_2,n')|\\
\nonumber&&+ \sum_{n_1\in W,|n_1-n|\geq \frac{M}{10}\atop n_2\in \Lambda\backslash W} e^{-\bar\rho|n-n_2|} |G_{\Lambda}(n_2,n')|\\
\label{rsi8}&\leq&  (2N+1)^{2d}\sup_{n_2\in \Lambda\backslash W}  e^{-(\bar\rho-\frac{O(1)}{\sqrt{M_0}})|n-n_2| }|G_{\Lambda}(n_2,n')|,
\end{eqnarray}
where the third inequality holds because of $\bar{\rho}\leq \frac{4}{5}\rho$ and   $|n-n_2|\geq \frac{M}{2}$.

Iterating \eqref{rsi8} until $|n_2-n'|\leq N^{\frac{1}{2}}$ (but at most $\frac{2|n-n'|}{M_0}$ times), we have $|n-n'|\geq\frac{N}{10}$,
\begin{eqnarray*}
  |G_{\Lambda}(n,n')| &\leq &  (2N+1)^{ \frac{O(|n-n'|)}{M_0}} e^{-(\bar\rho-\frac{O(1)}{\sqrt{M_0}}) (|n-n'|-N^{\frac{1}{2}}) }e^{\sqrt{N}}\\
   &\leq&  e^{-(\bar{\rho}-\frac{O(1)}{M_0^{1/2}})|n-n'|}  .
\end{eqnarray*}

\end{proof}

\begin{lem}[Several variables matrix-valued Cartan estimate]\label{mcl}
Let $T(x)$ be a self-adjoint $N\times N$ matrix function of a parameter $x\in[-\delta,\delta]^{J}$ ($J\in\mathbb{Z}^+$) satisfying the following conditions:
\begin{itemize}
\item[(i)] $T(x)$ is real analytic in $x\in [-\delta,\delta]^{J}$ and has a holomorphic extension to
\begin{equation*}
\mathcal{D}_{\delta,\delta_1}=\left\{x=(x_i)_{1\leq i\leq J}\in\mathbb{C}^{J}: \sup_{1\leq i\leq J}|\Re x_i|\leq\delta,\sup_{1\leq i\leq J}|\Im{x_i}|\leq \delta\right\}
\end{equation*}
satisfying
\begin{equation}\label{mc1}
\sup_{x\in \mathcal{D}_{\delta,\delta}}\|T(x)\|\leq B_1, B_1\geq 1.
\end{equation}
\item[(ii)]  For all $x\in[-\delta,\delta]^{J}$, there is subset $V\subset [1,N]$ with
\begin{equation*}|V|\leq M,\end{equation*}
and
\begin{equation}\label{mc2}
\|(R_{[1,N]\setminus V}T(x)R_{[1,N]\setminus V})^{-1}\|\leq B_2, B_2\geq 1.
\end{equation}
\item[(iii)]
\begin{equation}\label{mc3}
\mathrm{mes}\{x\in[-{\delta}, {\delta}]^{J}: \ \|T^{-1}(x)\|\geq B_3\}\leq 10^{-3J}J^{-J}\delta^J(1+B_1)^{-J}(1+B_2)^{-J}.
\end{equation}
Let
\begin{equation}\label{mc4}
0<\epsilon\leq (1+B_1+B_2)^{-10 M}.
\end{equation}
\end{itemize}
Then
\begin{equation}\label{mc5}
\mathrm{mes}\left\{x\in\left[-{\delta}/{2}, {\delta}/{2}\right]^{J}:\  \|T^{-1}(x)\|\geq \epsilon^{-1}\right\}\leq C\delta^Je^{-c\left(\frac{\log \epsilon^{-1}}{M\log(B_2+B_3)}\right)^{1/J}},
\end{equation}
where $C=C(J,B_1), c=c(J,B_1)>0$.
\end{lem}
\begin{proof}
The proof is similar to that of the case $J=1$ in Chapter 14 of \cite{bbook} (see also Remark 3 there).  We use the higher dimensional Cartan sets techniques of \cite{gs}. For convenience, we give the details in the Appendix.
\end{proof}

%
%

\begin{thm}\label{thmcar}
Under the assumptions of Theorem \ref{mul}, let $\omega\in \Omega_{N_2}\cap \Omega_{N_1}$.
We assume for some $x=(x^j, x_j^\neg)\in\T^b$, there exist $N\in[\frac{1}{4}N_3^{c_3}, N_3^{c_4}]$ and $\bar{\Lambda}\subset \Lambda\in \mathcal{E}_{N}$ with ${\rm diam}(\bar \Lambda)\leq 10N^{\frac{1}{10d}}$ such that,
for any $ k\in  \Lambda \backslash  \bar{\Lambda}$,
there exists some $k\in W\in \mathcal{E}_{N_1}$,$ W\subset \Lambda\backslash  \bar{\Lambda}$ such that ${\rm dist}(k,\Lambda \backslash  \bar{\Lambda}\backslash W)\geq \frac{N_1}{2},$ and $x+k\omega\mod\Z^b\notin X_{N_1}$.
Let
\begin{equation*}
  Y=\{y\in \R^{b_j}: |y-x^j|\leq e^{-\rho N_1},\|G_{\Lambda}(E;(y,x_j^\neg))\|\geq e^{\sqrt{N}}\}.
\end{equation*}
Then
\begin{equation}\label{Apr25}
\mathrm{mes}( Y)\leq e^{-N^{1/3b_j}}.
\end{equation}

\end{thm}

\begin{proof}
Without loss of generality, we  assume   $j=1$. Fix $x^1\in \T^{b_1}$ and $x_1^\neg\in\T^{b-b_1}$.



 Let $\mathcal{D}$ be the $e^{-\rho N_1}$ neighbourhood  of $x^1$ in the complex plane, i,e.,
\begin{equation*}
 \mathcal{D}=\{z\in \mathbb{C}^{b_1}:\  |\Im z|\leq e^{-\rho N_1}, |\Re z-x^1|\leq e^{-\rho N_1}\}.
\end{equation*}
By the assumption of Theorem \ref{thmcar}, one has for all $k\in \Lambda\backslash \bar{\Lambda}$ and $Q_{N_1}\in \mathcal{E}_{N_1}^0$,
\begin{eqnarray}
&&\|G_{Q_{N_1}}(E;x+k\omega)\|\leq e^{\sqrt{N_1}},\label{Apr20}\\
&&|G_{Q_{N_1}}(E;x+k\omega)(n,n')|\leq e^{-\bar{\rho}|n-n'|}\ {\rm for}\ |n-n'|\geq\frac{N_1}{10}.\label{Apr21}
\end{eqnarray}
By standard perturbation arguments\footnote{See e.g. the proof of Theorem \ref{thmini}.},  \eqref{Apr20} and \eqref{Apr21}, we have  for any $y\in \mathcal{D}$, $Q_{N_1}\in \mathcal{E}_{N_1}^0$,
and  $k\in\Lambda\backslash\bar\Lambda$,
\begin{eqnarray}
&&\|G_{Q_{N_1}}(E;(x^1+y,x_1^\neg)+k\omega)\|\leq 2e^{\sqrt{N_1}},\label{Apr22}\\
&&|G_{Q_{N_1}}(E;(x^1+y,x_1^\neg)+k\omega)(n,n')|\leq 2e^{-\bar{\rho}|n-n'|}\ {\rm for}\ |n-n'|\geq\frac{N_1}{10}.\label{Apr23}
\end{eqnarray}
 Substituting   $\Lambda$ with $\Lambda\backslash\bar{\Lambda}$ in Lemma \ref{res1}, one has for any $y\in \mathcal{D}$,
\begin{eqnarray}
\label{bs3}\|G_{\Lambda\setminus \bar\Lambda}(E;(x^1+y,x_1^\neg))\|\leq e^{2\sqrt{N_1}}.
\end{eqnarray}

We want to use  Lemma \ref{mcl}.
For this purpose, let  $$T(y)=\lambda^{-1}{H}_{\Lambda}((x^1+y,x_1^\neg))-{E},J=b_1,\delta=e^{-\rho N_1}.$$
Now we are in the position to check the assumptions of Lemma \ref{mcl}.
Obviously, $B_1=O(1)$ since $\lambda>1$ and $E$ is bounded.

Let $V=\bar{\Lambda}$.
By \eqref{bs3},  one has
\begin{eqnarray}\label{mb}
 M=|\bar\Lambda|\leq 30^dN^{{1}/{10}},
B_2=e^{2\sqrt{N_1}}.
\end{eqnarray}

By the fact that  the Green's functions satisfy property P and \eqref{Gstar},
 one has that
both \eqref{ldt1} and \eqref{ldt2} hold at scale $N_2$ for all $y$ except a   set  of $y\in \T^{b_1}$ with measure less than $e^{-{N_2^{c_1}}}$.
It implies both \eqref{ldt1} and \eqref{ldt2} holds at scale $N_2$ for all $x+k\omega$ with $|k| \leq N_3$ except a    set
of measure less than $(2N_3+1)^de^{-{N_2^{c_1}}}$.

Applying Lemma \ref{res1} with $M_0=M_1=N_2$ and \eqref{GGshift}, one has
\begin{eqnarray*}
 \|T^{-1}(y)\|\leq 4(2N_2+1)^de^{\sqrt{N_2}}\leq 4e^{2\sqrt{N_2}}=:B_3,
\end{eqnarray*}
except  on a   set of $y\in \T^{b_1}$
with measure less than $(2N_3+1)^de^{-{N_2^{c_1}}}$.

Since $N_2=N_1^{\frac{2}{c_1}}$,
direct computation shows that
\begin{equation*}
10^{-3b_1}b_1^{-b_1}\delta_1^{b_1}(1+B_1)^{-b_1}(1+B_2)^{-b_1}\geq e^{-{N_2^{c_1}}/{2}}.
\end{equation*}
This verifies (iii) in Lemma \ref{mcl}.

For $\epsilon=e^{-\sqrt{N}}$, by (\ref{mb}), one has
  $$\epsilon<(1+B_1+B_2)^{-10M}.$$

By (\ref{mc5}) of Lemma \ref{mcl},
\begin{equation}\label{y0}
\mathrm{mes}( Y)\leq C e^{-c\left(\frac{\sqrt{N}}{N_2N^{{1}/{10}}}\right)^{1/b_1}}\leq e^{-N^{{1}/{3b_1}}}.
\end{equation}

%
\end{proof}

\begin{thm}\label{mul1}
Let $c_1,c_2,c_3,c_4, N_1,N_2,N_3,\Omega_{3}$ be given by Theorem \ref{mul}, so in particular, Green's functions satisfy property P at $N_1,N_2$ with parameters $(c_1,\bar{\rho})$.  Then for all $N_3\leq N\leq N_{3}^2$, Green's functions satisfy property P at size ${N}$ with parameters $(c_1,\bar{\rho}-\frac{O(1)}{N_1^{1/2}})$ and   $\Omega_{N}=\Omega_{3}\cap\Omega_{N_2}$, where $O(1)$ only depends on $d$.
\end{thm}

\begin{proof}

We fix ${N}\in[{N}_3,{N}_3^2]$ and $Q_{N}\in\mathcal{E}_{N}^0$.  Let $\omega\in \Omega_{N_3}$.

 For any $n\in Q_{N}$,
replacing $x$ with $x+n\omega$ in Theorem \ref{mul},  there exits     ${N}_3^{c_3}< \bar{N}< {N}_3^{c_4}$ such that,
for all  $k\in (n+ \Lambda)\backslash (n+\bar\Lambda)$,  $ x+k\omega \mod \Z^b\notin X_{N_1}$, where
\begin{equation}\label{Apr3}
  \Lambda=[-\bar{N},\bar{N}]^d,\bar\Lambda=[-\bar{N}^{\frac{1}{10d}},\bar{N}^{\frac{1}{10d}}]^{d},
\end{equation}
and $n+\Lambda$, $n+\bar{\Lambda}$  are the shift of $\Lambda$ and $\bar{\Lambda}$ by $n$.

We are going to possibly shrink the $n+\Lambda$ a little bit so that it is in $  Q_{N}$.
More precisely,
we claim that for any $n\in Q_{N}$, there exist
\begin{equation}\label{Gjuly91}
\frac{1}{4}N_3^{c_3}\leq \tilde{N} \leq N_3^{c_4}
,
\end{equation}
$\Lambda_{\rm new} \in \mathcal{E}_{ \tilde{N}}$ and $\bar{\Lambda}_{\rm new}$, such that
\begin{equation}\label{Apr0}
  \Lambda_{\rm new} \subset \Lambda,   \bar\Lambda\subset \bar\Lambda_{\rm new},
\end{equation}

\begin{equation}\label{Apr1}
 n\in \Lambda_{\rm new} \subset Q_{N}, {\rm dist}(n,  Q_{N}\backslash \Lambda_{\rm new})\geq \frac{ \tilde{N}}{2}
\end{equation}
and
\begin{equation}\label{Apr2}
{\rm Diam}(\bar{\Lambda}_{\rm new}) \leq 4 \tilde{N}^{\frac{1}{10d}}.
\end{equation}
Also  for any $ k\in  \Lambda_{\rm new}\backslash  \bar{\Lambda}_{\rm new}$,
there exists some $\mathcal{E}_{N_1} \ni W \subset \Lambda_{\rm new}\backslash  \bar{\Lambda}_{\rm new}$ such that
\begin{equation}\label{Apr5}
  {\rm dist}(k,\Lambda_{\rm new}\backslash  \bar{\Lambda}_{\rm new}\backslash W)\geq \frac{N_1}{2}.
\end{equation}

We split the proof  into three cases.

Case 1: $n+\Lambda\subset Q_{N}$. In this case, let $\Lambda_{\rm new}=n+\Lambda$ and  $\bar{\Lambda}_{\rm new}=n+\bar{\Lambda}$. See Case 1 of Fig.1.

Case 2:   $(n+\Lambda)\cap (\Z^d \backslash Q_{N})$ is  non-empty and ${\rm dist}(n+\bar{\Lambda},\partial Q_{N})\geq 2N_1$. See Case 2 of Fig.1.
In this case, let $\bar{\Lambda}_{\rm new}=n+\bar{\Lambda}$ (the red square). By shrinking  $n+\Lambda$ a little bit, we can obtain
proper ${\Lambda}_{\rm new}\subset (n+\Lambda)\cap Q_{N}$ satisfying \eqref{Apr1}.  Since ${\rm dist}(n+\bar{\Lambda},\partial Q_{N})\geq 2N_1$, we can also
  guarantee \eqref{Apr5} holds.

\begin{center}
\begin{tikzpicture}[scale=0.618]
\draw[](0,0)rectangle (8,8);



\draw [fill=green](0.5,0.5) rectangle (3.5,3.5);
\draw [fill=red](1.5,1.5) rectangle (2.5,2.5);


\draw [](4.6,1.2) rectangle (8.2,4.8);
\draw [fill=green](5.2,1.8) rectangle (7.6,4.2);
\draw [fill=red](5.9,2.5) rectangle (6.9,3.5);

\draw (2,0.25) node[]   {\bf Case 1};
\draw (6,0.9) node[]   {\bf Case 2};
\draw (6.0,2.2) node[]   {$\Lambda_{\rm new}$};
\draw (4,0) node[below]   {Fig.1};

\end{tikzpicture}
\end{center}
Case 3:   $(n+\Lambda)\cap (\Z^d \backslash Q_{N})$ is non-empty and ${\rm dist}(n+\bar{\Lambda},\partial Q_{N})\leq 2N_1$.
%
In this case,
making    $(n+\bar{\Lambda})\cap Q_N $ possibly larger,  we  obtain $\bar{\Lambda}_{\rm new}\subset Q_N$. We  can also make sure for any $ k\in  Q_N\backslash  \bar{\Lambda}_{\rm new}$,
there exists some $  W\in \mathcal{E}_{N_1}\subset Q_N\backslash  \bar{\Lambda}_{\rm new}$
\begin{equation}\label{Apr13}
  {\rm dist}(k, Q_N\backslash  \bar{\Lambda}_{\rm new}\backslash W)\geq \frac{N_1}{2}.
\end{equation}
See  Fig.2. For B, $\bar{\Lambda}_{\rm new}=(n+\bar{\Lambda})\cap Q_N $ (the red part).
For A and C, $(n+\bar{\Lambda})\cap Q_N$ is the red part, and $\bar{\Lambda}_{\rm new}$ is union of the red part and the blue part.
Shrinking    $n+\Lambda$, we can obtain
proper $ {\Lambda} _{\mathrm{new}}$  satisfying \eqref{Apr1}. This implies  \eqref{Apr5} by \eqref{Apr13}.

\begin{center}
\begin{tikzpicture}[scale=0.618]
\draw[](0,0)rectangle (8,8);

\draw [](5.5,8) rectangle (6.5,8.5);
\draw [fill=green](5.2,6.4) rectangle (6.8,8);
\draw [fill=red](5.5,7.5) rectangle (6.5,8);
\draw [](4.2,6.2) rectangle (7.8,9.8);



\draw [](0.5,6) rectangle (3.5,9);
\draw [fill=green](1.2,8) rectangle (2.8,6.6);
\draw [fill=red](1.7,7.2) rectangle (2.3,7.8);
\draw [fill=blue](1.7,7.8) rectangle (2.3,8);

\draw [](-0.85,-0.85) rectangle (2.15,2.15);
\draw [fill=green](0,0) rectangle (2,2);
\draw [fill=red](0.2,0.2) rectangle (1.1,1.1);
\draw [fill=blue](0,0) rectangle (0.2,1.1);
\draw [fill=blue](0,0) rectangle (1.1,0.2);



\draw (6,5.9) node[]   {\bf B};
\draw (6,7) node[]   {$\Lambda_{\rm new}$};
\draw (1.8,2.5) node[]   {\bf C};
\draw (1.3,1.6) node[]   {$\Lambda_{\rm new}$};

\draw (2,5.7) node[]   {\bf  A};
\draw (1.8,6.9) node[]   {$\Lambda_{\rm new}$};
\draw (4,-1.2) node[]   {Fig.2: Case 3};
\end{tikzpicture}
\end{center}

Fix $x_j^\neg$. Divide $\T^{b_j}$ into $e^{2 b_j\rho N_1}$ cubes of size $ e^{-\rho N_1}$.

Applying Theorem \ref{thmcar} in each cube, (\eqref{Apr0}, \eqref{Apr2} and \eqref{Apr5} ensure we can use Theorem \ref{thmcar}),
 there exists a set $Y_{\tilde{N}} (x_j^\neg)$ such that
\begin{equation}\label{Apr6}
\mathrm{mes}(Y_{\tilde{N}}(x_j^\neg))\leq e^{2 b_j\rho N_1}e^{-\tilde{N}^{\frac{1}{3b_j}}},
\end{equation}
and for $x=(x^j,x_j^\neg)$ with $x^j\notin Y_{\tilde{N}}(x_j^\neg)$,
\begin{equation}\label{Apr7}
\|G_{\Lambda_{\rm new}}(E;(x^j,x_j^\neg))\|\leq e^{\sqrt{\tilde{N}}}.
\end{equation}

Setting $M_0=N_1$, $\Lambda=\Lambda_{\rm new}$ and $ \Lambda_1=\bar{\Lambda}_{\rm new}$ in Theorem \ref{res2} (\eqref{ldt1}, \eqref{ldt2}, \eqref{Apr0}, \eqref{Apr2}, \eqref{Apr5} and
\eqref{Apr7} ensure we can use Theorem \ref{res2}), we have for such $x$,
\begin{equation}\label{Apr8}
  |G_{\Lambda_{\rm new}}(E;x)(n,n')|\leq  e^{- (\bar{\rho}-\frac{O(1)}{N_1^{1/2}})|n-n'|}\  {\mathrm{for} \ |n-n'|\geq \frac{\tilde{N}}{10}}.
\end{equation}

Let
\begin{equation}\label{Apr9}
 B_{N}(x_j^\neg)=  \bigcup_{\frac{1}{4}N_3^{c_3}\leq \tilde{N}\leq N_3^{c_4} }Y_{\tilde{N}}(x_j^\neg).
\end{equation}
By \eqref{Apr6}, \eqref{Apr9} and  since $c_1=c_3/4\tilde{b}$,
one has for any $j$ and $x_j^\neg\in \T^{b-b_j}$,
\begin{equation}\label{Apr12}
\mathrm{mes}( {B}_{{N}}(x_j^\neg))\leq  e^{-{N}^{c_1}}.
\end{equation}
Suppose  $x^j\notin {B}_{{N}}(x_j^\neg)$. Applying $\Lambda=Q_{N}$, $M_0= \frac{1}{4}N_3^{c_3}$ and $M_1=  N_3^{c_4}$
in  Lemma \ref{res1} since $N\in[N_3,N_3^2]$ (\eqref{Gjuly91}, \eqref{Apr1}, \eqref{Apr7} and \eqref{Apr8} ensure the assumption of Lemma  \ref{res1}),
 one has
\begin{equation}\label{Apr10}
\|G_{Q_N}(E;x)\|\leq 4(2 N_3^{c_4}+1)^de^{\sqrt{N_3^{c_4}}}\leq e^{\sqrt{N}}.
\end{equation}
Applying  $\Lambda=Q_{N}\in\mathcal{E}_N^{0}$, $M_0= \frac{1}{4}N_3^{c_3}$ and $ \Lambda_1=\emptyset$ in Theorem \ref{res2}, by \eqref{Apr7}, \eqref{Apr8} and \eqref{Apr10},
we have
\begin{equation}\label{Apr11}
  |G_{Q_N}(E;x)(n,n')|\leq  e^{- (\bar{\rho}-\frac{O(1)}{N_1^{1/2}})|n-n'|}\  {\mathrm{for} \ |n-n'|\geq \frac{{N}}{10}}.
\end{equation}

Let
\begin{equation*}
  X_N=\{x\in \T^b:(E,x) \text { is not } (\bar{\rho}-\frac{O(1)}{N_1^{1/2}},N) \text{ good }\}, \Omega_N=\Omega_3\cap \Omega_{N_2}.
\end{equation*}
The theorem follows from \eqref{Apr11}, \eqref{Apr10} and \eqref{Apr12}.
\end{proof}

\section{Large deviation theorem for Green's functions  and proof of Theorem \ref{mthm}}

The main result of this section is the following large deviation theorem (LDT) for Green's functions.
\begin{thm}[LDT]\label{ldt}
There exist constants $ \gamma=\gamma(b,d)\in(0,1)$,  $N_0=N_0(v,\rho,b,d)$ and $ \lambda_0=\lambda_0(v,\rho,b,d)$,
such that  for all $N\geq N_0$ and $\lambda\geq \lambda_0$, the Green's functions satisfy property P with parameters $(\gamma,\frac{\rho}{2})$  at size $N$,
and the corresponding   semi-algebraic set $\Omega_N$ satisfying
\begin{equation*}
   \mathrm{mes}(\T^b\backslash \cap_{N\geq N_0} \Omega_N) \to 0,
 \end{equation*}
 as $\lambda\to \infty$.
\end{thm}
Compactness arguments and  Theorem 8 in \cite{PSS} immediately imply
\begin{lem}[{\L}ojasiewicz  type Lemma]\label{loj}
For $E\in\mathbb{R},\delta>0$, define $$X:=\{x\in\mathbb{T}^b:\  |v(x)-E|<\delta\}.$$ Then there are  constants $C(v),a(v)>0$ such that
\begin{equation}
\sup_{1\leq j\leq d, x_j^\neg\in \T^{b-b_j}}\mathrm{mes}( X(x_j^\neg))\leq C(v)\delta^{a(v)}.
\end{equation}
\end{lem}
\begin{thm}\label{thmini}
Let ${X}$ be as in Lemma \ref{loj} and
\begin{equation}\label{lowdel}
X_{N}:=\bigcup_{|n|\leq {N}}\left\{x:\ x+n{{\omega}}\mod\mathbb{Z}^b\in{X}\right\}.
\end{equation}
Then   we have
\begin{equation}\label{imb}
\sup_{1\leq j\leq d, x_j^\neg\in \T^{b-b_j}}\mathrm{mes}( X_N(x_j^\neg))\leq C(v)(2N+1)^{d}\delta^{a(v)}.
\end{equation}
Moreover, if
\begin{equation}\label{ibc}
\lambda\geq 2\delta^{-1}(2N+1)^{d},
\end{equation}
then for any $x\notin X_{N},\ \omega\in\T^b$, we have for $Q_N\in\mathcal{E}_{N}^0$,
\begin{eqnarray}
&& \|G_{Q_N}(E;x)\|\leq 2\delta^{-1}, \label{numan1}\\
&& |G_{Q_N}(E;x)(n,n')|\leq 2\delta^{-1}e^{-{\rho}|n-n'|} .\label{numan}
\end{eqnarray}
\end{thm}
\begin{proof}
The  bound (\ref{imb}) follows  from Lemma \ref{loj} immediately.

Let $x\notin X_N$ and fix   $Q_N\in\mathcal{E}_{N}^0$. Let $A=\lambda^{-1}R_{Q_N}SR_{Q_N}$, with the kinetic term $S$ being given by \eqref{GO}.
Let $B$ be diagonal part of  the restriction of $\lambda^{-1}H-E$ on $Q_N$, namely, $$B=R_{Q_N}(v(x+n\omega)\delta_{nn'}-E\delta_{nn'})R_{Q_N}.$$
By \eqref{lowdel}, one has
$$\min_{n\in Q_N}\left|v(x+n\omega)-{E}\right|\geq \delta.$$
It leads to
\begin{equation}\label{boub}
\| B^{-1}\| \leq \delta^{-1}.
\end{equation}
Since $|S(n,n')|\leq e^{-\rho|n-n'|}$ for all $n,n'$, by Lemma \ref{schur} again, one has for $N\geq N(\rho,d)$,
\begin{equation}\label{boua}
\|A\|\leq \lambda^{-1}\sup_{n\in Q_N}\sum_{n'\in Q_N}e^{-\rho|n-n'|}\leq\lambda^{-1} (2N+1)^d.
\end{equation}

By (\ref{ibc}),
$$\|AB^{-1}\|\leq \frac{1}{2}.$$
Combining with
\eqref{boub} and \eqref{boua}, we have the following Neumann series expansion
\begin{equation}\label{nsa}
G_{Q_N}= B^{-1}\sum_{s\geq 0} (-AB^{-1})^s.
\end{equation}
Thus one has
\begin{equation}\label{numan2}
\|G_{Q_N}\|\leq ||B^{-1}\| \frac{1}{1-||AB^{-1}\|}\leq 2\delta^{-1}.
\end{equation}
It implies \eqref{numan1}.
In particular,
\eqref{numan} is also true for $n=n'$.

For $n\neq n'$, by \eqref{GO}, \eqref{nsa} and the fact that $B$ is diagonal, we have
\begin{eqnarray*}
 | G_{Q_N}(n,n')| &\leq & \|B^{-1}\| \sum_{s\geq 1\atop |k_i|\leq N}\lambda^{-s}\delta^{-s} e^{-\rho|n-k_1|-\rho|k_1-k_2|-\cdots|k_{s-1}-n'|} \\
   &\leq & \delta^{-1} e^{-\rho|n-n'|} \sum_{s\geq 1} (2N+1)^{sd} \lambda^{-s}\delta^{-s}\\
   &\leq &2 \delta^{-1} e^{-\rho|n-n'|} ,
\end{eqnarray*}
where the last inequality holds by \eqref{ibc}.

\end{proof}

\begin{proof}[\bf Proof of Theorem \ref{ldt}]
Denote  the relation between $N_1$ and ${N}_3$ in Theorem \ref{mul} by $f$, i.e., $f(x)=e^{x^{c_1}}$. Let $N_0=N_0(v,\rho,b,d)$ be sufficiently large. Denote by $f^{(n)}(x)$  the $n$th iteration of $f$, namely, $f^{(n)}(x)=f(f(f(\cdots x\cdots)))$. Let $g(x)=f^2(x)$.  Clearly, $g(x)\geq f(x+1)$ for large $x$.

By letting  $\delta=\frac{1}{2}e^{ -\bar{N}^{1/2}}$
and  Theorem \ref{thmini}, since $c_1<1/2$,
  the Green's functions satisfy property P with parameters $(c_1,\frac{4\rho}{5})$    for   $N_0\leq N\leq \bar{N}$ and     $\Omega_N= \T^b$
if  $  \lambda \geq 4 e^{\bar{N}^{1/2}} (2\bar{N}+1)^{d}$.

  Theorem \ref{mul1} allows us to proceed from scales $N$, $N^{\frac{2}{c_1}}$ to scales $[f(N),g(N)]$. Since we want to cover all scales, our initial step will consist of property P at the interval of  scales $[N_1,f(N_1)]$. For this reason, we need to take $N_1=\log\log \lambda$.

Initial step:  For large $\lambda$,  the  Green's functions satisfy property P with parameters $(c_1,\rho_0)$    for  all $N_0\leq N\leq  g(\log\log \lambda)$ and     $\Omega_N= \T^b$,
where $\rho_0=\frac{4\rho}{5}$.

Let
\begin{equation}\label{derhoi}
\rho_i=\frac{4\rho}{5}-\sum_{j=1}^i\frac{O(1)}{f^{(j)}(\log\log \lambda)^{1/2}}.
\end{equation}

Applying Theorem \ref{mul1} to $N_1=\log\log \lambda,\log\log \lambda+1,\log\log \lambda+2, \cdots ,f(\log\log\lambda)$,
 the  Green's functions satisfy property P with parameters $(c_1,\rho_1)$    for  all $g(\log\log\lambda)\leq N\leq  g(f(\log\log \lambda))$
 since $g(x)\geq f(x+1)$.
 Moreover,
 \begin{eqnarray}
    {\rm mes}\left(\bigcap_{N=\log\log \lambda}^{f(\log\log \lambda)} \Omega_{N}\right) &\geq& 1-\sum _{N=\log\log\lambda}^{f(\log\log \lambda)} \frac{1}{ f(N)^{c_3}}\nonumber \\
    &\geq& 1-\sum _{N=\log\log\lambda}^{f(\log\log \lambda)}  \frac{1}{N^5}.\label{meome}
 \end{eqnarray}

 Applying Theorem \ref{mul1} to $N_1=f(\log\log \lambda),f(\log\log \lambda)+1,f(\log\log \lambda)+2, \cdots ,f^{(2)}(\log\log\lambda)$,
 the  Green's functions satisfy property P with parameters $(c_1,\rho_2)$    for  all $g(f(\log\log \lambda))\leq N\leq  g(f^{(2)}(\log\log \lambda))$.
 Moreover,
 \begin{equation*}
  {\rm mes}\left(\bigcap_{N=f(\log\log \lambda)+1}^{f^{(2)}(\log\log \lambda)} \Omega_{N}\right)\geq 1-\sum _{N=f(\log\log \lambda)+1}^{f^{(2)}(\log\log \lambda)}  \frac{1}{N^5}.
 \end{equation*}

 By induction, we  have
  the  Green's functions satisfy property P with parameters $(c_1,\rho_i)$    for  $g(f^{(i-1)}(\log\log \lambda))\leq N\leq  g(f^{(i)}(\log\log \lambda))$, $i=1,2,\cdots$.
 Moreover,
 \begin{equation}\label{meome1}
  {\rm mes}\left(\bigcap_{N=f^{(i-1)}(\log\log \lambda)+1}^{f^{(i)}(\log\log \lambda)} \Omega_{N}\right)\geq 1-\sum _{N=f^{(i-1)}(\log\log \lambda)+1}^{f^{(i)}(\log\log \lambda)}  \frac{1}{N^5}.
 \end{equation}

 Now Theorem \ref{ldt}  follows from \eqref{derhoi} and \eqref{meome1}.
\end{proof}
\begin{proof}[\bf Proof of Theorem \ref{mthm}]
With Theorem \ref{ldt} at hand, the proof Theorem \ref{mthm} is rather standard. We refer the readers to \cite[Section 3]{bbook} or \cite[Section 6]{bgs} for details.
\end{proof}

\appendix
\section{}
In the following, we will prove the several variables matrix-valued Cartan estimate, i.e., Lemma \ref{mcl}. The proof is similar to that in \cite{bbook,bkick}. Before going to the details, we recall some useful lemmas. The first result is the standard Schur's complement theorem. For convenience, we include a proof here.
\begin{lem}\label{sl}
Let $T$ be the matrix
\begin{equation*}
T=\left(
\begin{array}{cc}
T_1&T_2\\
T_2^t&T_3
\end{array}
\right),
\end{equation*}
where $T_1$ is an invertible $n\times n$ matrix , $T_2$ is an $n\times k$ matrix and $T_3$ is a $k\times k$ matrix. Let
$$S=T_3-T_2^tT_1^{-1} T_2.$$
Then $T$ is invertible if and only if $S$ is invertible, and
\begin{equation}\label{sl1}
\|S^{-1}\|\leq \|T^{-1}\|\leq C(1+\|T_1^{-1}\|)^2(1+\|S^{-1}\|),
\end{equation}
where $C$ depends only on $\|T_2\|$.
\end{lem}
\begin{proof}
It is easy to check that
\begin{equation}\label{appaug1}
T=\left(
\begin{array}{cc}
T_1&T_2\\
T_2^t&T_3
\end{array}
\right)=\left(
\begin{array}{cc}
I&0\\
T_2^{t}T_1^{-1}&I
\end{array}
\right)\left(
\begin{array}{cc}
I&T_2\\
0&S
\end{array}
\right)\left(
\begin{array}{cc}
T_1&0\\
0&I
\end{array}
\right).
\end{equation}
It implies $T$ is invertible if and only if $S$ is invertible and also
the second inequality of \eqref{sl1}.  
By  \eqref{appaug1}, one has
\begin{eqnarray*}
T^{-1}&=&  \left(
\begin{array}{cc}
T_1&0\\
0&I
\end{array}
\right)^{-1}\left(
\begin{array}{cc}
I& T_2 \\
0& S
\end{array}
\right)^{-1}  \left(
\begin{array}{cc}
I&0\\
T_2^{-1}T_1^{-1}&I
\end{array}
\right)^{-1}\\
  &=& \left(
\begin{array}{cc}
T_1^{-1}&0\\
0&I
\end{array}
\right)\left(
\begin{array}{cc}
I&-T_2S^{-1}\\
0& S^{-1}
\end{array}
\right)   \left(
\begin{array}{cc}
I&0\\
-T_2^{-1}T_1^{-1}&I
\end{array}
\right)\\
&=&
 \left(
\begin{array}{cc}
\star&\star\\
\star& S^{-1}
\end{array}
\right).
\end{eqnarray*}
implying the first inequality of \eqref{sl1}.
\end{proof}

We then introduce the higher dimensional Cartan sets Lemma of  Goldstein-Schlag  \cite{gs}.  We denote by $\mathcal{D}(z,r)$  the standard disk on $\mathbb{C}$ of center  $z$ and radius  $r>0$.

\begin{lem}\cite[Lemma 2.15]{gs}\label{svcl}
Let $f(z_1,\cdots,z_J)$ be an analytic function defined in a ploydisk $\mathcal{P}=\prod\limits_{1\leq i\leq J}\mathcal{D}(z_{i,0},1/2)$ and $\phi=\log|f|$. Let $\sup\limits_{\underline{z}\in \mathcal{P}}\phi(\underline{z})\leq M,m\leq \phi(\underline{z}_0)$, $\underline{z}_0=(z_{1,0},\cdots,z_{J,0})$. Given $F\gg1$, there exists a set $\mathcal{B}\subset\mathcal{P}$  such that
\begin{equation}\label{cal1}
\phi(\underline{z})>M-C(J)F(M-m),\  \mathrm{for}\ \forall\ \underline{z}\in \prod\limits_{1\leq i\leq J}\mathcal{D}(z_{i,0},1/4)\setminus \mathcal{B},
\end{equation}
and
\begin{equation}\label{cal2}
 \mathrm{mes}(\mathcal{B}\cap\mathbb{R}^J)\leq  C(J)e^{-F^{1/J}}.
\end{equation}
\end{lem}

\begin{proof}[\textbf{Proof of Lemma \ref{mcl}}]
The proof is similar to that of Proposition 14.1 in \cite{bbook} in case $J=1$ and Lemma 1.43 in \cite{bkick} without explicit bounds.  
In the following proof, $C=C(B_1,J)$  and $c=c(B_1,J)$.

 Let
$$\mu=10^{-2}{J^{-1}}\delta(1+B_1)^{-1}(1+B_2)^{-1}.$$ Fix
$$x_0\in\left[-\delta/2, \delta/2\right]^{J}$$
and consider $T(z)$ with $|z-x_0|=\sup\limits_{1\leq i\leq J}|z_i-x_{0,i}|<\mu$.  Thanks to Cauchy's estimate and (\ref{mc1}), one obtains for $|z-x_0|<\mu$,
$$\|{\partial_{z_i} T(z)}\|\leq \frac{4 B_1}{\delta},i=1,2,\cdots, J,$$
which implies
$$\|T(z)-T(x_0)\|\leq \frac{4JB_1\mu}{\delta}\leq 25^{-1}(1+B_2)^{-1}.$$
From the assumption (ii) of Lemma \ref{mcl}, we can find $V=V(x_0)$ so that $|V|\leq M$ and \eqref{mc2} is satisfied. Denote by $V^c=[1,N]\setminus V$. Thus using the standard Neumann series argument and (\ref{mc2}), one has
\begin{equation}\label{acb}
\|(R_{V^c}T(z)R_{V^c})^{-1}\|\leq 2B_2\ \mathrm{for}\  |z-x_0|<\mu.\end{equation}
We define for $|z-x_0|<\mu$ the analytic  self-adjoint function
\begin{equation}\label{scf}
S(z)=R_{V}T(z)R_{V}-R_{V}T(z)R_{V^c}(R_{V^c}T(z)R_{V^c})^{-1}R_{V^c}T(z)R_{V}.
\end{equation}
Then by (\ref{acb}) and (\ref{scf}), we have
\begin{equation}\label{sb}
\|S(z)\|\leq 3B_1^2B_2.
\end{equation}
Recalling Lemma \ref{sl}, if $S(z)$ is invertible, so is $T(z)$ and by (\ref{sl1}),
\begin{equation}\label{sib}
\|S^{-1}(z)\|\leq C\|T^{-1}(z)\|\leq CB_2^2(1+\|S^{-1}(z)\|).
\end{equation}
For $x\in\mathbb{R}^{J}$, 
one has
\begin{equation}\label{dets}
||S(x)||^M\geq |\det S(x)|=\prod_{\lambda\in \sigma(S(x))}|\lambda|\geq \|S^{-1}(x)\|^{-M}.
\end{equation}
By (\ref{sb}), one has
\begin{equation}\label{sib1}
\|S^{-1}(x)\|\leq \frac{\|S(x)\|^{M-1}}{|\det S(x)|}\leq \frac{(3B_1^2B_2)^M}{|\det S(x)|}.
\end{equation}
Let
$$\phi(z)=\log|\det S(x_0+\mu z)|,\  |z|<1.$$
Then by (\ref{dets}) and (\ref{sb}),
\begin{equation}\label{pub}
\sup_{|z|<1}\phi(z)\leq C M\log B_2.
\end{equation}
By (\ref{mc3}) and the definition of $\mu$, there is some $x_1$ with $|x_0-x_1|<\mu/10$  such that
\begin{equation}
\|T^{-1}(x_1)\|\leq B_3.
\end{equation}
Hence by (\ref{sib}), $\|S^{-1}(x_1)\|\leq CB_3$, and from (\ref{dets}),
\begin{equation}\label{plb}
\phi(a)\geq -CM\log B_3,
\end{equation}
where $a=\frac{x_1-x_0}{\mu}$, so $ |a|<1/10$.      Let
\begin{equation*}
\mathcal{P}=\prod_{1\leq i\leq J}\mathcal{D}(a_i,{1}/{2}).
\end{equation*}
Then  one has
\begin{equation*}
\sup_{z\in\mathcal{P}}\phi(z)\leq C M\log B_2, \phi(a)\geq- CM\log B_3.
\end{equation*}
Applying Lemma \ref{svcl} and recalling (\ref{cal1}), (\ref{cal2}), for any $F\gg1$, there is some set $\mathcal{B}\subset \prod\limits_{1\leq i\leq J}\mathcal{D}(a_i,{1}/{4})$ with
\begin{equation}\label{plbb}
\phi(z)\geq -CF M\log(B_2+B_3)\ \mathrm{for}\ z\in \prod\limits_{1\leq i\leq J}\mathcal{D}(a_i,{1}/{4})\setminus \mathcal{B},
\end{equation}
and
\begin{equation}\label{bm}
\mathrm{mes}(\mathcal{B}\cap\mathbb{R}^{J})\leq Ce^{-F^{1/J}}.
\end{equation}
For $0<\epsilon<1$, let
\begin{equation*}
F =\frac{-c\log \epsilon}{ M\log(B_2+B_3)}.
\end{equation*}
Then by (\ref{plbb}) and (\ref{bm}),
\begin{eqnarray}
\nonumber&&\mathrm{mes}\left\{x\in\mathbb{R}^{J}: \ |x-x_1|<\mu/4\ \mathrm{and}\ |\det S(x)|\leq \epsilon\right\}\\
\nonumber&&\ \ \ \ \ \  = \mu^{{J}}\mathrm{mes}\left\{x\in\mathbb{R}^{J}:\  |x-a|<1/4\ \mathrm{and}\ \phi(x)\leq \log \epsilon\right\}\\
\nonumber&&\ \ \ \ \ \ \leq C\mu^{{J}} e^{-F^{1/J}}.
\end{eqnarray}
Since  $|x_0-x_1|<\mu/10$, we have
\begin{equation}\label{x0b}
\mathrm{mes}\left\{x\in\mathbb{R}^{J}:\  |x-x_0|<\mu/8\ \mathrm{and}\ |\det(S(x))|\leq \epsilon\right\}\leq C \mu^{{J}} e^{-c\left(\frac{\log \epsilon^{-1}}{M\log(B_2+B_3)}\right)^{1/J}}.
\end{equation}
Recalling (\ref{sib}), (\ref{sib1}) and (\ref{mc4}), one has for $|x-x_0|<\mu/8$ and $|\det S(x)|\geq \epsilon$,
\begin{equation}\label{aib}
\|T^{-1}(x)\|\leq C(1+B_2^2)(1+\epsilon^{-1}(3B_1^2{B_2})^M)\leq C\epsilon^{-2}.
\end{equation}
Covering $[-\frac{\delta}{2},\frac{\delta}{2}]^{J}$ by cubes of side  $\mu/4$, and combining (\ref{x0b}) and (\ref{aib}), one has
$$\mathrm{mes}\left\{x\in\left[-\delta/2, \delta/2\right]^{J}:\  \|T^{-1}(x)\|\geq \epsilon^{-2}\right\}\leq C\delta^{J}e^{-c\left(\frac{\log \epsilon^{-1}}{M\log(B_2+B_3)}\right)^{1/J}}.$$
\end{proof}

\section*{Acknowledgments}
   We are grateful to Jean Bourgain for his encouragement. This research was
 supported by NSF DMS-1401204, DMS-1901462,  and DMS-1700314.

\end{document}